\documentclass[11pt,a4paper,reqno]{amsart}

\usepackage{a4wide}
\usepackage[T1]{fontenc}
\usepackage{lmodern}
\usepackage{amsmath,amssymb,amsthm}
\usepackage{url}
\usepackage[colorlinks=true,linkcolor=blue,citecolor=blue,urlcolor=blue]{hyperref}
\usepackage[nameinlink,noabbrev]{cleveref}
\usepackage{aliascnt}

\theoremstyle{plain}
\newtheorem{maintheorem}{Theorem}

\newaliascnt{maincorollary}{maintheorem}
\newtheorem{maincorollary}[maincorollary]{Corollary}
\aliascntresetthe{maincorollary}
\newtheorem{theorem}{Theorem}[section]

\newaliascnt{lemma}{theorem}
\newtheorem{lemma}[lemma]{Lemma}
\aliascntresetthe{lemma}

\newaliascnt{proposition}{theorem}
\newtheorem{proposition}[proposition]{Proposition}
\aliascntresetthe{proposition}

\newaliascnt{corollary}{theorem}

\aliascntresetthe{corollary}

\theoremstyle{definition}

\theoremstyle{remark}

\crefname{maintheorem}{theorem}{theorems}
\Crefname{maintheorem}{Theorem}{Theorems}
\crefname{maincorollary}{corollary}{corollaries}
\Crefname{maincorollary}{Corollary}{Corollaries}

\newcommand{\R}{\mathbb R}
\newcommand{\C}{\mathbb C}
\newcommand{\Hh}{\mathbb H}
\newcommand{\N}{\mathbb N}
\newcommand{\F}{\mathbb F}
\newcommand{\norm}[1]{\lVert #1\rVert}
\newcommand{\abs}[1]{\lvert #1\rvert}
\newcommand{\inner}[2]{\langle #1,#2\rangle}
\newcommand{\cP}{\mathcal P}
\newcommand{\cE}{\mathcal E}
\newcommand{\Lip}{\operatorname{Lip}}
\newcommand{\Aut}{\operatorname{Aut}}
\newcommand{\End}{\operatorname{End}}
\newcommand{\Hom}{\operatorname{Hom}}
\newcommand{\Sym}{\operatorname{Sym}}
\newcommand{\vol}{\operatorname{vol}}
\newcommand{\Gr}{\operatorname{Gr}}
\newcommand{\Id}{\mathrm{Id}}

\subjclass[2020]{Primary 46C15, 52A21; Secondary 46A04, 46A13, 55R10, 55M25}
\keywords{Banach's isometric conjecture, Minkowski's functional, star body, quaternionic Banach space, metric Fr\'echet space, graded Fr\'echet space, principal bundle, Brouwer degree}

\begin{document}

\title[Banach's isometric conjecture]{Banach's Isometric Conjecture over the Complex Field}

\author[A. Acuaviva]{Antonio Acuaviva}
\address[A.~Acuaviva]{School of Mathematical Sciences, Lancaster University, Lancaster LA1 4YF, United Kingdom}
\email{ahacua@gmail.com}

\author[T. Kania]{Tomasz Kania}
\address[T.~Kania]{Mathematical Institute\\Czech Academy of Sciences\\\v Zitn\'a 25 \\115 67 Praha 1\\Czech Republic  and  Institute of Mathematics and Computer Science\\ Jagiellonian University\\ {\L}ojasiewicza 6, 30-348 Krak\'{o}w, Poland
}
\email{kania@math.cas.cz, tomasz.marcin.kania@gmail.com}
\thanks{RVO: 67985840.}

\date{}

\begin{abstract}
We complete Banach's isometric conjecture over the complex field. More precisely, if \(X\) is a complex normed space and, for some \(2\leqslant n<\dim_{\C}X\), all its \(n\)-dimensional complex subspaces are isometric as metric spaces, then the norm is induced by a Hermitian inner product. We also prove the quaternionic counterpart. The central geometric argument first treats real star bodies without convexity or central symmetry; applied to circled complex or quaternionic bodies, it shows that mutually real-linearly equivalent hyperplane sections force the ambient body to be a Hermitian ellipsoid. The proof adapts the bundle-degree mechanism introduced by Lu and Yang for the real case. Finally, we obtain extensions to absolutely homogeneous functions, graded Fr\'echet spaces, metrisable locally convex spaces, and compatible translation-invariant metrics.
\end{abstract}

\maketitle

\section{Introduction}\label{sec:introduction}

The question of how much the geometry of a normed space is determined by its subspaces of one fixed finite dimension goes back to Banach. He asked whether a real or complex Banach space \(X\) must be a Hilbert space whenever, for some fixed integer \(2\leqslant n<\dim X\), all its \(n\)-dimensional subspaces are linearly isometric \cite[Chapter~XII]{Banach1932}. This is Banach's isometric conjecture.

Dvoretzky's theorem settles the conjecture for every fixed \(n\geqslant2\) in infinite-dimensional real Banach spaces \cite{Dvoretzky1959}, and Milman's complex version gives the corresponding complex result \cite{Milman1971}. The finite-dimensional history is more intricate. Auerbach, Mazur, and Ulam proved the real case \(n=2\) in 1935 \cite{AuerbachMazurUlam1935}. Gromov proved the conjecture for every even \(n\) over \(\R\) and \(\C\). For odd \(n\), he also proved it over \(\R\) when \(\dim_{\R}X\geqslant n+2\), and over \(\C\) when \(\dim_{\C}X\geqslant2n\) \cite{Gromov1967}.

The bundle approach was subsequently developed by Bor, Hern\'andez-Lamoneda, Jim\'enez-Desantiago, and Montejano, who proved the real case \(n\equiv1\pmod4\), \(n\geqslant5\), with the possible exception of \(n=133\) \cite{BorEtAl2021}. Ivanov, Mamaev, and Nordskova settled the low-dimensional case \(n=3\), \(\dim_{\R}X=4\) \cite{IvanovMamaevNordskova2023}. Most recently, Lu and Yang proved the real conjecture for every odd \(n\), thereby completing the real problem \cite{LuYang2026}. Their paper is explicitly confined to real spaces. The bundle-degree mechanism used here is theirs, and their construction in turn has the work of Bor et al. as a direct antecedent.

Over \(\C\), Bracho and Montejano proved the cases \(n\equiv1\pmod4\) \cite{BrachoMontejano2021}. Consequently, before the present work the only residual complex range was
\[
 n\equiv3\pmod4,
 \qquad n<\dim_{\C}X<2n.
\]
The usual two-vector reduction shows that it is enough to consider \(\dim_{\C}X=n+1\): once every \((n+1)\)-dimensional subspace is Hermitian, the parallelogram identity holds for every pair of vectors. We prove precisely this codimension-one case and hence close the residual range. The quaternionic statement appears to have no antecedent in the literature known to us.

For completeness, Zhang's recent preprint claims a proof of the full finite-dimensional real problem by means of the John ellipsoid \cite{Zhang2025}. Its Theorem~3.1 states that an origin-symmetric star body with ellipsoidal hyperplane sections is an ellipsoid. This overlaps the star-body setting but not the implication proved here: our sections are assumed only to be mutually real-linearly equivalent, their ellipsoidality is part of the conclusion, and no central symmetry is imposed. We next establish the complex and quaternionic results considered here, together with several Fréchet-space extensions. To the best of our knowledge, these are not covered by previous results.

The result follows from a statement in which neither convexity nor the triangle inequality is assumed. Throughout the paper, \(\F\in\{\C,\Hh\}\). A function \(p\colon X\to[0,\infty)\) is \emph{definite} if \(p(x)=0\) only for \(x=0\), and it is \emph{\(\F\)-absolutely homogeneous} if \(p(xa)=p(x)\abs a\) for \(x\in X\) and \(a\in\F\).

\begin{maintheorem}\label{thm:functional-intro}
Let \(X\) be a right \(\F\)-vector space, let \(2\leqslant n<\dim_{\F}X\), and let \(p\colon X\to[0,\infty)\) be definite and \(\F\)-absolutely homogeneous. Suppose that
\begin{enumerate}
\renewcommand{\labelenumi}{\textup{(\roman{enumi})}}
\renewcommand{\theenumi}{(\roman{enumi})}

\item\label{it:1thA} on every \((n+1)\)-dimensional \(\F\)-subspace \(Y\) of \(X\), the restriction of \(p\) to an auxiliary Euclidean unit sphere is Lipschitz; and

\item\label{it:2thA} whenever \(E\) and \(F\) are \(n\)-dimensional \(\F\)-subspaces of \(X\), there is a real-linear bijection \(T\colon E\to F\) such that \(p(Tx)=p(x)\) for every \(x\in E\).
\end{enumerate}
Then there is a positive-definite \(\F\)-Hermitian form \(h\) on \(X\) such that \(p(x)^2=h(x,x)\) for every \(x\in X\).
\end{maintheorem}

No completeness assumption occurs in \Cref{thm:functional-intro}, and only \((n+1)\)-dimensional subspaces are used. The Lipschitz condition is independent of the auxiliary Euclidean structure. Indeed, if \(\abs\cdot_1\) and \(\abs\cdot_2\) are two Euclidean norms on \(Y\), then the radial identification
\[
 S(Y,\abs\cdot_2)\longrightarrow S(Y,\abs\cdot_1),
 \qquad \theta\longmapsto\frac{\theta}{\abs\theta_1},
\]
is smooth and bi-Lipschitz, and absolute homogeneity expresses \(p\) on the first sphere as the product of a smooth positive factor and its restriction to the second. For a norm, condition~\ref{it:1thA} is automatic. The Mazur--Ulam theorem also shows that the isometries in the following consequence need not be assumed linear, or compatible with multiplication by complex or quaternionic scalars.

\begin{maincorollary}\label{cor:normed-intro}
Let \(X\) be an \(\F\)-normed space, and suppose that, for some \(2\leqslant n<\dim_{\F}X\), all \(n\)-dimensional \(\F\)-subspaces of \(X\) are bijectively isometric as metric spaces. Then the norm of \(X\) is induced by a positive-definite \(\F\)-Hermitian form. Consequently, the norm completion of \(X\) is an \(\F\)-Hilbert space, and \(X\) itself is an \(\F\)-Hilbert space when it is complete.
\end{maincorollary}

For \(\F=\C\), the Banach-space case of \Cref{cor:normed-intro} completes the complex conjecture; for \(\F=\Hh\), it gives its quaternionic counterpart. Notice that the proof is dimension-free and obtains the inner product before completion. In particular, the infinite-dimensional quaternionic conclusion does not come from Milman's complex theorem by restriction of scalars: the quaternionic \(n\)-subspaces form only a proper subfamily of the complex \(2n\)-subspaces.

The geometric statement from which these conclusions follow is likely to be useful independently.

\begin{theorem}\label{thm:hyperplane-intro}
Let \(n\geqslant2\), let \(V\) be a right \(\F\)-vector space of dimension \(n+1\), and let \(K\subseteq V\) be an \(\F\)-circled star body with positive Lipschitz radial function. If all central \(\F\)-hyperplane sections of \(K\) are real-linearly equivalent, then \(K\) is the unit ball of a positive-definite \(\F\)-Hermitian form.
\end{theorem}

In the complex case, \Cref{thm:hyperplane-intro} strengthens the geometric result of Bracho and Montejano in two independent respects: a circled star body replaces a symmetric convex body, and the equivalences between its complex hyperplane sections need only be real-linear. Thus convexity is removed, while circledness is retained for the final scalar reconstruction. The underlying real ellipsoid theorem, \Cref{thm:spherical-intro}, requires neither convexity nor central symmetry.

The remaining main results extend this rigidity beyond Banach spaces, first to graded Fr\'echet spaces and then to vector spaces equipped with a compatible translation-invariant metric. The projective-limit description in \Cref{thm:reduced-frechet} is standard once the defining seminorms have been shown to be Hilbertian; the substantive assertion is their eventual Hilbertianity. The simultaneous eigenvalue argument in \Cref{thm:simultaneous-collapse} gives the stronger conclusion that an entire separating family collapses to one Hermitian norm.

The homogeneous formulation also applies to Fr\'echet spaces, provided that one fixes either a grading or a compatible translation-invariant metric. Following the graded and multi-seminormed viewpoint of \cite{BargetzKakolKubis2017,KawachLopezAbad2022}, we shall call an increasing defining sequence \((p_k)_{k\in\N}\) \emph{reduced \(n\)-isometric} if, whenever \(n<\dim_{\F}(X/\ker p_k)\), all \(n\)-dimensional subspaces of the normed quotient \(X/\ker p_k\) are bijectively isometric. The isometry may depend on \(k\).

\begin{maintheorem}\label{thm:reduced-frechet}
Let \(X\) be a Hausdorff Fr\'echet space over \(\F\), where \(\dim_{\F}X>n\geqslant2\). The following assertions are equivalent.
\begin{enumerate}
\item \(X\) admits an increasing defining sequence of Hilbertian seminorms.
\item \(X\) admits a reduced \(n\)-isometric increasing defining sequence.
\end{enumerate}
For every sequence \((p_k)_{k\in\N}\) as in \emph{(2)}, there is \(k_0\in\N\) such that the cofinal tail \((p_k)_{k\geqslant k_0}\) consists of Hilbertian seminorms, and
\[
X\cong\varprojlim_{k\geqslant k_0}\widehat{X/\ker p_k},
\]
where the spaces on the right are \(\F\)-Hilbert spaces and the canonical bonding maps have dense range.
\end{maintheorem}

This conclusion does not imply normability. A stronger conclusion is available when the same map acts at every seminorm level, up to a level-dependent positive factor.

\begin{maintheorem}\label{thm:simultaneous-collapse}
Let \(\F\in\{\R,\C,\Hh\}\). Let \(X\) be a right \(\F\)-vector space, let \(\mathcal Q\) be a separating family of \(\F\)-seminorms, and let \(2\leqslant n<\dim_{\F}X\). Suppose that, for every pair of \(n\)-dimensional subspaces \(E,F\subseteq X\), there is an \(\F\)-linear isomorphism \(T\colon E\to F\) such that, for each \(p\in\mathcal Q\), there is \(c_p(T)>0\) for which \(p(Tx)=c_p(T)p(x)\) for every \(x\in E\). Then there is a positive-definite \(\F\)-Hermitian form \(h\) on \(X\) such that every member of \(\mathcal Q\) is either zero or a positive multiple of \(x\mapsto h(x,x)^{1/2}\). Thus \(\mathcal Q\) generates a norm topology; if this topology is complete, then \(X\) is an \(\F\)-Hilbert space.
\end{maintheorem}

A compatible translation-invariant metric gives a third formulation. For such a metric \(d\), put \(B_r^d=\{x:d(x,0)\leqslant r\}\).

\begin{maintheorem}\label{thm:metric-frechet}
Let \(\F\in\{\R,\C,\Hh\}\). Let \(X\) be a Fr\'echet right \(\F\)-space, let \(2\leqslant n<\dim_{\F}X\), and let \(d\) be a compatible translation-invariant metric. Suppose that for every pair of \(n\)-dimensional \(\F\)-subspaces \(E,F\subseteq X\), there is an \(\F\)-linear bijection \(T\colon E\to F\) satisfying
\[
d(Tx,Ty)=d(x,y)
\qquad(x,y\in E).
\]
Suppose also that there is a set \(\mathcal R\subseteq(0,\infty)\), with \(0\in\overline{\mathcal R}\), for which one of the following assertions holds.
\begin{enumerate}
\item Every \(B_r^d\), \(r\in\mathcal R\), is \(\F\)-absolutely convex.
\item For every \(r\in\mathcal R\) and every finite-dimensional \(\F\)-subspace \(W\subseteq X\), the section \(B_r^d\cap W\) is a compact \(\F\)-circled star body in \(W\) with positive Lipschitz radial function.
\end{enumerate}
If \(\F=\R\) and assertion~\emph{(2)} is used, assume in addition that \(n\) is odd.
Then a positive-definite \(\F\)-Hermitian form on \(X\) induces the original topology, and \(X\) is an \(\F\)-Hilbert space. If every closed centred \(d\)-ball is \(\F\)-absolutely convex, the corresponding norm may be chosen so that
\[
d(x,y)=\phi\bigl(\norm{x-y}\bigr),
\]
where \(\phi\colon[0,\infty)\to[0,\infty)\) is non-decreasing and subadditive, vanishes at \(0\), and is strictly positive on \((0,\infty)\).
\end{maintheorem}

The proof starts from the Lipschitz bundle-degree argument of Lu and Yang \cite[Sections~2--4]{LuYang2026}, developed there for origin-symmetric convex bodies and their hyperplane sections. Mutually equivalent sections determine a principal bundle of exact section maps. After reduction to the identity component, a self-map of the parameter sphere of positive degree annihilates the bundle class. A Lipschitz trivialisation then gives an exact family to which the signed degree formula applies. It follows that two consecutive even powers of Minkowski's functional of a reference section are polynomials; unique factorisation makes its square quadratic. The real ambient dimensions and the relevant homotopy-group indices have the required parity over both \(\C\) and \(\Hh\). A final polarisation argument recovers the scalar structure. We give the details needed to pass from origin-symmetric convex bodies and hyperplane sections to possibly non-convex star bodies and the more general section fields used below.

\section{Preliminaries}\label{sec:preliminaries}

We write \(\N=\{1,2,\ldots\}\). Euclidean structures, adjoints, tensor products, and measures are real unless a scalar field is indicated. Vector spaces over \(\F\) are right modules. An \(\F\)-Hermitian form is conjugate-linear in the first variable and right-linear in the second; for \(\Hh\) we use the conventions of \cite{AlpayColomboSabadini2015}. A set \(K\) is \emph{\(\F\)-circled} if \(Ka=K\) whenever \(a\in\F\) and \(\abs a=1\).

An \(\F\)-seminorm \(p\) is \emph{Hilbertian} if \(p(x)^2=h(x,x)\) for a positive-semidefinite \(\F\)-Hermitian form \(h\). Its radical is then \(\ker p\). When \(h\) is positive definite, we also call \(p\) a \emph{Hermitian norm}. Completeness is not part of this terminology. A graded Fr\'echet space is a Fr\'echet space with a fixed increasing defining sequence of seminorms. A subset \(C\) of a right \(\F\)-vector space is \emph{\(\F\)-absolutely convex} if \(xa+yb\in C\) whenever \(x,y\in C\) and \(\abs a+\abs b\leqslant1\).

Let \(V\) be a finite-dimensional real Euclidean space, with closed unit ball \(B_V\) and unit sphere \(S(V)\). A \emph{star body about the origin} is a set
\begin{equation}\label{eq:star-body}
 K=\{t\theta:\theta\in S(V),\ 0\leqslant t\leqslant\rho_K(\theta)\},
\end{equation}
where \(\rho_K\colon S(V)\to(0,\infty)\) is continuous. \emph{Minkowski's functional} of \(K\) is \(p_K(x)=\inf\{t>0:x\in tK\}\). Thus \(p_K(0)=0\) and \(p_K(x)=\abs x/\rho_K(x/\abs x)\) for \(x\neq0\); it need not be a norm.

If \(\rho_K\) is Lipschitz, then so is \(p_K\). Indeed, put \(f=1/\rho_K\), and let \(M=\sup f\) and \(L=\Lip(f)\). If \(0<\abs x\leqslant\abs y\), then \(\abs x\abs{x/\abs x-y/\abs y}\leqslant2\abs{x-y}\), and consequently
\begin{equation}\label{eq:minkowski-lipschitz}
\abs{p_K(x)-p_K(y)}\leqslant(M+2L)\abs{x-y}.
\end{equation}
The cases in which one of the vectors is zero follow directly. No convexity is required for this estimate.

Let \(E\) be an \(m\)-dimensional real Euclidean space and let \(S\subseteq E\) be a star body. Its \emph{normalised covariance} is the positive operator \(C(S)=\vol(S)^{-1}\int_S x\otimes x\,dx\). For every real-linear isomorphism \(T\), change of variables gives
\begin{equation}\label{eq:covariance-transform}
C(TS)=TC(S)T^*.
\end{equation}
In particular, applying \(C(S)^{-1/2}\) puts \(S\) in \emph{covariance position}, where \(C(S)= \Id \).

The covariance here is deliberately taken about the prescribed origin, not about the barycentre. The origin is part of the star-body structure and is fixed by every linear equivalence under consideration. Positivity of \(C(S)\) uses only that \(S\) contains a neighbourhood of the origin: for \(v\neq0\), the function \(x\mapsto(x,v)_{\R}\) is nonzero on an open subset of \(S\), and hence \((C(S)v,v)_{\R}>0\). Neither central symmetry nor convexity is involved.

For \(j\geqslant1\), put \(M_j(S)=\vol(S)^{-1}\int_Sx^{\otimes j}\,dx\in\Sym^j(E)\). Then \(M_j(TS)=T^{\otimes j}M_j(S)\). Polar coordinates give
\begin{equation}\label{eq:polar-moments}
 \vol(S)=\frac1m\int_{S(E)}\rho_S(\theta)^m\,d\theta,
 \qquad
 M_j(S)=\frac{m}{m+j}\,
 \frac{\int_{S(E)}\rho_S(\theta)^{m+j}\theta^{\otimes j}\,d\theta}
      {\int_{S(E)}\rho_S(\theta)^m\,d\theta}.
\end{equation}
It follows that, on any family with common positive inner and finite outer radii, each fixed moment depends Lipschitz-continuously on the radial function in the uniform norm. 

The next lemma combines two results of Lu and Yang \cite[Lemmas~3.1--3.2]{LuYang2026}. In their centrally symmetric setting, the odd moments vanish. In the present setting, the odd moments need not vanish, and the same finite-generation argument simply selects whichever even or odd degrees are required. Neither the covariance normalisation nor the moment and stabiliser arguments use convexity. Thus the proof applies to Lipschitz star bodies about the origin.

\begin{lemma}\label{lem:finite-moments}
Suppose that \(C(S)= \Id\), and put \(G=\Aut_{\R}(S)=\{g\in\operatorname{GL}_{\R}(E):gS=S\}\). Then:
\begin{enumerate}
\renewcommand{\labelenumi}{\textup{(\roman{enumi})}}
\renewcommand{\theenumi}{\roman{enumi}}
\item\label{item:finite-moments-compact} \(G\) is a compact subgroup of \(O(E)\).
\item\label{item:finite-moments-stabiliser} There are positive integers \(j_1,\ldots,j_s\) such that \(G\) is exactly the stabiliser in \(O(E)\) of the tuple \(M(S)=(M_{j_1}(S),\ldots,M_{j_s}(S))\).
\item\label{item:finite-moments-orbit} The orbit map \(O(E)/G\to O(E)M(S)\), given by \(QG\mapsto QM(S)\), is a diffeomorphism onto a compact embedded submanifold, and its inverse is locally Lipschitz for the ambient Euclidean distance.
\end{enumerate}
\end{lemma}

\begin{proof}
    The proof of~\textup{(\ref{item:finite-moments-compact})} is verbatim that of \cite[Lemma~3.1]{LuYang2026}: from \(C(S)=\Id\), every \(g\in G\) satisfies \(gg^*=\Id\), and \(G\) is a closed subgroup of \(O(E)\), hence compact.

    For~\textup{(\ref{item:finite-moments-stabiliser})}, the same argument requires only one minor modification. In the centrally symmetric setting of Lu and Yang the odd moments vanish, whereas here \(S\) need not satisfy \(S=-S\). We therefore retain the full sequence \((M_j(S))_{j\geqslant1}\). Every member of \(G\) fixes every moment of \(S\). Conversely, suppose that \(Q\in O(E)\) fixes \(M_j(S)\) for every \(j\geqslant1\). The normalised uniform measures on \(S\) and \(QS\) then have the same integral against every polynomial. By the Stone--Weierstrass theorem, they agree on \(C(S\cup QS)\), and hence are equal. Since their supports are \(S\) and \(QS\), respectively, it follows that \(QS=S\).
    
    Finally, the equations \(Q^{\otimes j}M_j(S)=M_j(S)\), \(j\geqslant1\), are polynomial equations in the matrix entries of \(Q\). The Hilbert basis theorem therefore reduces this family to finitely many degrees \(j_1,\ldots,j_s\), whose common stabiliser in \(O(E)\) is still exactly \(G\). This proves~\textup{(\ref{item:finite-moments-stabiliser})}.
    
    With this finite moment tuple in hand, the proof of~\textup{(\ref{item:finite-moments-orbit})} is verbatim that of \cite[Lemma~3.2]{LuYang2026}.
\end{proof}

\section{Polynomial rigidity}\label{sec:polynomial-rigidity}

We use the standard conventions for Brouwer degree and for the degree of maps between oriented spheres, as in \cite[Section~2.3]{LuYang2026}. The following theorem records the form of the degree calculation of Lu and Yang \cite[Lemmas~4.1--4.6]{LuYang2026} needed below. Their degree argument extends to star bodies and positively homogeneous functions. The only point in their proof where convexity is used is in Lemma~4.4, through their Lemma~2.1, to ensure that \(\partial K\) has zero ambient measure. For a star body with Lipschitz radial function, this follows instead because \(\theta\mapsto\rho_K(\theta)\theta\) is a Lipschitz map from \(S(V)\) onto \(\partial K\). Central symmetry is not used in the degree calculation; its earlier role was only to annihilate odd moments, a point already handled in \Cref{lem:finite-moments} by retaining moments of every positive degree. We include the details for completeness.

\begin{theorem}\label{thm:exact-family-intro}
Let \(V\) be an oriented real Euclidean space of dimension \(N>0\), let \(E\) be a nonzero finite-dimensional real vector space equipped with an auxiliary Euclidean structure, let \(K\subseteq V\) be a star body, and let \(p\colon E\to[0,\infty)\). Suppose that:
\begin{enumerate}
    \renewcommand{\labelenumi}{\textup{(\roman{enumi})}}
    \renewcommand{\theenumi}{\roman{enumi}}
    \item\label{item:exact-family-parity} \(N\) is even;
    \item\label{item:exact-family-radial} the radial function of \(K\) is Lipschitz;
    \item\label{item:exact-family-functional} \(p\) is positive away from the origin and positively homogeneous;
    \item\label{item:exact-family-maps} there are a continuous map \(\varphi\colon S(V)\to S(V)\), with \(\deg\varphi\neq0\), and a Lipschitz map \(A\colon S(V)\to\Hom_{\R}(E,V)\);
    \item\label{item:exact-family-identities} the identities
    \begin{equation}\label{eq:exact-family}
    (A_zy,\varphi(z))_{\R}=0
    \quad\hbox{and}\quad
    p_K(A_zy)=p(y)
    \qquad(z\in S(V),\ y\in E)
    \end{equation}
    hold.
\end{enumerate}
Then \(p^{N+2k}\) is a real homogeneous polynomial for every \(k\geqslant0\). In particular, \(p^2\) is a positive-definite real quadratic form.
\end{theorem}

\begin{proof}
    Put \(D=\deg\varphi\), which is nonzero by~\textup{(\ref{item:exact-family-maps})}. Extend \(A\) homogeneously to \(B_V\) by setting \(\mathcal A(0)=0\) and \(\mathcal A(tz)=tA_z\) for \(0<t\leqslant1\) and \(z\in S(V)\). The Lipschitz property in~\textup{(\ref{item:exact-family-maps})} shows that this extension is Lipschitz. Indeed, if \(L_0=\sup_z\norm{A_z}\), \(L_1=\Lip(A)\), \(x=\lambda z\), \(x'=\mu w\), and \(0\leqslant\lambda\leqslant\mu\leqslant1\), then
    \[
        \norm{\lambda A_z-\mu A_w}
        \leqslant \lambda L_1\abs{z-w}+L_0\abs{\lambda-\mu}
        \leqslant (2L_1+L_0)\abs{x-x'}.
    \]
    Here we used \(\lambda\abs{z-w}\leqslant2\abs{x-x'}\) and \(\abs{\lambda-\mu}\leqslant\abs{x-x'}\); the same estimate covers the origin.
    
    Fix \(y\neq0\) and put \(r=p(y)>0\), where positivity follows from~\textup{(\ref{item:exact-family-functional})}. Define \(F_y(x)=\mathcal A(x)y\) for \(x\in B_V\) and \(w_y(z)=A_zy/\abs{A_zy}\) for \(z\in S(V)\). The second identity in~\textup{(\ref{item:exact-family-identities})} ensures that \(A_zy\neq0\). By the first identity in~\textup{(\ref{item:exact-family-identities})}, the unit vectors \(\varphi(z)\) and \(w_y(z)\) are orthogonal, so
    \[
        (z,t)\longmapsto
        \cos(\pi t/2)\varphi(z)+\sin(\pi t/2)w_y(z)
        \qquad(0\leqslant t\leqslant1)
    \]
    is a homotopy through the unit sphere. Hence \(\deg w_y=D\).
    
    The second identity in~\textup{(\ref{item:exact-family-identities})} also gives \(A_zy\in r\partial K\). The radial homeomorphism \(\widehat\Psi_r(t\theta)=tr\rho_K(\theta)\theta\), for \(t\geqslant0\), is joined to the identity by the positive radial interpolation \(t\theta\mapsto t((1-s)+sr\rho_K(\theta))\theta\). It is therefore orientation preserving and maps \(B_V\) onto \(rK\). If \(Cw_y(tz)=tw_y(z)\) denotes the cone on \(w_y\), then \(F_y=\widehat\Psi_r\circ Cw_y\). The boundary characterisation and composition rule for Brouwer degree consequently give
    \begin{equation}\label{eq:cone-degree}
        \deg(F_y,\operatorname{int}B_V,v)=D
        \quad\hbox{for }v\in\operatorname{int}(rK),
    \end{equation}
    whereas the degree is zero for \(v\notin rK\), since \(F_y(B_V)\subseteq rK\). No smoothness of \(\partial K\) is involved.
    
    Let \(\Omega=\operatorname{int}B_V\). The coordinatewise McShane theorem \cite{McShane1934} extends \(F_y\) from the closed ball to a Lipschitz map \(\widetilde F_y\) on \(V\). The extension agrees with \(F_y\) on \(\Omega\), and therefore has the same derivative almost everywhere there. We spell out the measure-theoretic hypotheses of the signed degree formula. Every Lipschitz map between \(N\)-dimensional Euclidean spaces has the Lusin \((N)\) property: if \(A\) has measure zero, then the standard Hausdorff-measure estimate
    \[
        \mathcal H^N(\widetilde F_y(A))
        \leqslant \Lip(\widetilde F_y)^N\mathcal H^N(A)
    \]
    shows that its image has measure zero. Moreover, by~\textup{(\ref{item:exact-family-radial})}, \(\partial K\) is the Lipschitz image of the \((N-1)\)-sphere under \(\theta\mapsto\rho_K(\theta)\theta\), so it has \(N\)-dimensional measure zero. Since \(F_y(\partial\Omega)\subseteq r\partial K\), the image of the boundary is null. Rademacher's theorem provides the derivative almost everywhere; on the bounded set \(\Omega\), its determinant is integrable because \(\widetilde F_y\) is Lipschitz.
    
    Choose \(\chi\in C_c(V)\) which is one on a neighbourhood of \(rK\), and put \(g(v)=\chi(v)\abs v^{2k}\), interpreting \(\abs v^0\) as one. The function \(g\) is bounded and compactly supported. Thus all the hypotheses of the signed area--degree formula \cite[Remark~5.26(ii) and Theorem~5.27]{FonsecaGangbo1995} are satisfied by \(\widetilde F_y\), \(\Omega\), and \(g\). As \(\widetilde F_y=F_y\) on \(\Omega\), \eqref{eq:cone-degree} gives
    \begin{equation}
        \int_\Omega\abs{F_y(x)}^{2k}\det_{\R}DF_y(x)\,dx
        =D\int_{rK}\abs v^{2k}\,dv \notag
        =D\,p(y)^{N+2k}\int_K\abs u^{2k}\,du.
    \label{eq:signed-degree}
    \end{equation}
    
    Regard \(\mathcal A\colon\Omega\to\Hom_{\R}(E,V)\) as a Lipschitz map between Euclidean spaces. Rademacher's theorem supplies a single null set \(Z\subseteq\Omega\), independent of \(y\), outside which \(\mathcal A\) is differentiable. If \(h_1,\ldots,h_N\) is an oriented orthonormal basis of \(V\), then \(DF_y(x)[h_j]=(D\mathcal A(x)[h_j])y\) for \(x\notin Z\). Each column of the Jacobian is real-linear in \(y\), and hence its determinant is a homogeneous polynomial of degree \(N\) in \(y\). Moreover, \(\abs{F_y(x)}^{2k}=(\mathcal A(x)y,\mathcal A(x)y)_{\R}^{k}\) is a homogeneous polynomial of degree \(2k\). Both \(\mathcal A\) and \(D\mathcal A\) are essentially bounded, so every coefficient of the resulting polynomial integrand is bounded by an integrable constant on \(\Omega\). The left-hand side of \eqref{eq:signed-degree} is consequently a homogeneous polynomial \(P_k(y)\) of degree \(N+2k\). The constant \(D\int_K\abs u^{2k}\,du\) is nonzero because \(D\neq0\) by~\textup{(\ref{item:exact-family-maps})} and \(K\) has positive volume. Thus \(p^{N+2k}=P_k/(D\int_K\abs u^{2k}\,du)\) away from zero, and hence everywhere by~\textup{(\ref{item:exact-family-functional})}.
    
    It remains to use~\textup{(\ref{item:exact-family-parity})}. Write \(N=2a\), and let \(P=p^{2a}\) and \(Q=p^{2a+2}\), which are nonzero homogeneous polynomials by what we have just proved. Pointwise, and therefore as a polynomial identity, \(Q^a=P^{a+1}\). For each irreducible polynomial \(\pi\), unique factorisation gives \(a\nu_\pi(Q)=(a+1)\nu_\pi(P)\). Since \(a\) and \(a+1\) are coprime, there is an integer \(\ell_\pi\geqslant0\) such that \(\nu_\pi(P)=a\ell_\pi\) and \(\nu_\pi(Q)=(a+1)\ell_\pi\). Thus \(P\) divides \(Q\). The quotient \(R=Q/P\) is homogeneous of degree two and satisfies \(R(y)=p(y)^2\) for \(y\neq0\), hence also at zero. By~\textup{(\ref{item:exact-family-functional})}, \(R\) is positive away from zero, and therefore it is a positive-definite real quadratic form.
\end{proof}

\section{Spherical sections and exact families}\label{sec:exact-sections}

The preceding theorem turns a nonzero-degree exact family into an ellipsoid. We next show that mutually equivalent spherical sections provide such a family.

\begin{theorem}\label{thm:spherical-intro}
    Let \(V\) be an oriented real Euclidean space of even dimension \(N\geqslant4\), let \(1\leqslant d\leqslant N-1\), and let \(u\mapsto H_u\) be a smooth map from \(S(V)\) to the Grassmannian \(\Gr_d(V)\), with \(H_u\subseteq u^\perp\) for every \(u\). Let \(K\subseteq V\) be a star body with positive Lipschitz radial function. If the sections \(K\cap H_u\) are mutually real-linearly equivalent, then every one of them is a real ellipsoid. If, in addition, every pair of vectors in \(V\) belongs to some \(H_u\), then \(K\) is a real ellipsoid.
\end{theorem}

No injectivity of \(u\mapsto H_u\) is assumed or used. In the complex and quaternionic applications the parametrisation is necessarily redundant: \(H_{ua}=H_u\) whenever \(\abs a=1\). Working over the full sphere is precisely what provides the normal vector required by the degree argument.

For the remainder of this section, let \(\Sigma=S(V)\) and put \(S_u=K\cap H_u\). The construction is real; the complex and quaternionic structures enter only in \Cref{sec:scalar-rigidity}. We use the standard terminology for principal bundles, local sections, transition functions, pullbacks, and reductions of structure group, with the conventions of \cite[Section~2.2]{LuYang2026}. By a finite \emph{Lipschitz atlas} we mean a finite trivialising cover whose local sections and transition functions are Lipschitz.

Let us fix \(u_0\in\Sigma\), put \(E=H_{u_0}\), and choose a reference body \(S\subseteq E\) equivalent to every \(S_u\). After a real-linear change of coordinates, we may suppose that \(C(S)= \Id\). Put \(G=\Aut_{\R}(S)\) and, for \(u\in\Sigma\),
\[
 \cP_u=\{A\in\Hom_{\R}(E,V):A(E)=H_u,\ A(S)=S_u\}.
\]
Right composition by \(G\) is free and transitive on each \(\cP_u\). We now assemble these sets into a principal bundle.

\begin{proposition}\label{prop:exact-bundle}
    For every \(u_*\in\Sigma\), there is a connected neighbourhood \(U\) of \(u_*\) and a Lipschitz map \(u\mapsto A_u\in\Hom_{\R}(E,V)\) such that \(A_u(E)=H_u\) and \(A_u(S)=S_u\). Finitely many such maps may be chosen so that the transition functions into \(G\) are Lipschitz on their overlaps. Consequently, \(\cP=\coprod_{u\in\Sigma}\cP_u\to\Sigma\) is a principal \(G\)-bundle with a finite Lipschitz atlas.
\end{proposition}

\begin{proof}
    Choose an orthogonal map \(R_*\colon E\to H_{u_*}\), and denote by \(\Pi_u\) the orthogonal projection of \(V\) onto \(H_u\). For \(u\) near \(u_*\), set
    \[
        J_u=\Pi_uR_*\bigl(R_*^*\Pi_uR_*\bigr)^{-1/2}.
    \]
    At \(u=u_*\), the operator under the inverse square root is the identity. After shrinking the neighbourhood it remains positive definite, and \(\Pi_uR_*\) has range \(H_u\). Moreover,
    \(J_u^*J_u=(R_*^*\Pi_uR_*)^{-1/2}(R_*^*\Pi_uR_*)(R_*^*\Pi_uR_*)^{-1/2}=\Id\). Thus \(J_u\colon E\to H_u\) is an orthogonal map which depends smoothly on \(u\).
    
    Put \(\widetilde S_u=J_u^{-1}S_u\). Its radial function satisfies \(\rho_{\widetilde S_u}(\theta)=\rho_K(J_u\theta)\). On a relatively compact neighbourhood of \(u_*\), the maps \(J_u\) have uniformly bounded derivatives, and hence \(u\mapsto\rho_{\widetilde S_u}\) is Lipschitz in the uniform norm. These bodies have common positive inner and finite outer radii. It follows from \eqref{eq:polar-moments} that \(C_u=C(\widetilde S_u)\) depends Lipschitz-continuously on \(u\). Its eigenvalues remain in a compact subinterval of \((0,\infty)\), so the smooth functional calculus shows that \(C_u^{1/2}\) and \(C_u^{-1/2}\) are Lipschitz as well.
    
    For each \(u\), choose a real-linear isomorphism \(T_u\) with \(\widetilde S_u=T_uS\); no continuity of this choice is needed. Since \(C(S)= \Id\), \eqref{eq:covariance-transform} gives \(C_u=T_uT_u^*\). It follows that \(C_u^{-1/2}T_u\) is orthogonal, and therefore \(\widehat S_u=C_u^{-1/2}\widetilde S_u\) is an orthogonal image of \(S\). For each degree selected in \Cref{lem:finite-moments},
    \(M_j(\widehat S_u)=(C_u^{-1/2})^{\otimes j}M_j(\widetilde S_u)\). Hence the finite moment tuple \(u\mapsto M(\widehat S_u)\) is a Lipschitz map into the orbit \(O(E)M(S)\).
    
    By \Cref{lem:finite-moments}, the corresponding map into \(O(E)/G\) is Lipschitz. Since the quotient map \(O(E)\to O(E)/G\) admits smooth local sections \cite[Theorem~20.12, Corollary~21.6, and Theorems~21.10 and~4.26]{Lee2013}, after shrinking the neighbourhood we obtain a Lipschitz map \(u\mapsto Q_u\in O(E)\) such that \(Q_uS=\widehat S_u\). Hence \(A_u=J_uC_u^{1/2}Q_u\) depends Lipschitz-continuously on \(u\), has range \(H_u\), and maps \(S\) onto \(S_u\).
    
    Let us now choose a finite cover by connected chart domains whose closures lie in neighbourhoods on which the preceding construction is defined. On an overlap there is a unique \(t_{ji}(u)\in G\) such that \(A_j(u)=A_i(u)t_{ji}(u)\), and
    \[
        t_{ji}(u)=\bigl(A_i(u)^*A_i(u)\bigr)^{-1}A_i(u)^*A_j(u).
    \]
    The least singular value of \(A_i(u)\) is bounded away from zero on the chart closure. Matrix multiplication and inversion are Lipschitz on the resulting compact sets, and so every \(t_{ji}\) is Lipschitz on the whole overlap. Finally, use \((u,g)\mapsto A_i(u)g\) from \(U_i\times G\) onto \(\coprod_{u\in U_i}\cP_u\) as a local trivialisation. The transition identities make the induced topologies compatible and give the asserted principal bundle.
\end{proof}

Let \(G^\circ\) denote the identity component of \(G\). A global section need not exist. The obstruction can nevertheless be removed after passage to the identity component and pullback by a suitable self-map of the sphere.

\begin{lemma}\label{lem:annihilation}
    The bundle \(\cP\) contains a principal \(G^\circ\)-subbundle \(\cP^\circ\) with a finite Lipschitz atlas. Moreover, there are \(D\geqslant1\) and a smooth map \(\varphi\colon S^{N-1}\to S^{N-1}\) of degree \(D\) such that \(\varphi^*\cP^\circ\) is topologically trivial.
\end{lemma}

\begin{proof}
    The group \(G/G^\circ\) is finite, so \(\cP/G^\circ\to S^{N-1}\) is a finite covering. Since \(N\geqslant4\), the sphere is simply connected, and every component of this covering maps homeomorphically onto the base. Choose one component and let \(\cP^\circ\) be its inverse image in \(\cP\). On each connected chart the chosen component is represented by a fixed right coset \(g_iG^\circ\). Replacing the local section \(A_i\) by \(A_i g_i\) changes the transition function to \(t_{ji}^\circ=g_i^{-1}t_{ji}g_j\in G^\circ\). It is still Lipschitz, and hence the reduction has a finite Lipschitz atlas.
    
    The reduced bundle determines a class \(\alpha\in\pi_{N-1}(BG^\circ)\cong\pi_{N-2}(G^\circ)\); see \cite[Chapter~1, Theorem~5.3, and Chapter~4, Sections~8 and~11--13]{Husemoller1994}. Since \(G^\circ\) is connected, \(BG^\circ\) is simply connected, so the based class represents the corresponding free homotopy class without a residual fundamental-group action. A compact connected Lie group is finitely covered by the product of a torus and a compact simply connected semisimple group \cite[Section~2.9]{BorelHirzebruch1958}. Covering maps induce isomorphisms on homotopy groups in degrees at least two, the higher homotopy groups of a torus vanish, and Serre's theorem shows that the even-dimensional homotopy groups of the semisimple factor are finite \cite[Chapter~V, Section~3, Corollary~2]{Serre1953}. Since \(N-2\) is positive and even, \(\alpha\) has finite order. Choose \(D\geqslant1\) with \(D\alpha=0\).
    
    There is a based continuous self-map \(\varphi_0\) of \(S^{N-1}\) of degree \(D\). Approximate it, as an \(\R^N\)-valued map, by a smooth map \(g\) with \(\abs{g-\varphi_0}<1/2\), and put \(\varphi_1=g/\abs g\). Normalised straight-line interpolation gives a homotopy from \(\varphi_0\) to \(\varphi_1\). Postcomposing with a suitable element of \(SO(N)\) makes the map based without changing its degree; call the resulting map \(\varphi\). Precomposition by \(\varphi\) acts as multiplication by \(D\) on \(\pi_{N-1}(BG^\circ)\) \cite[Lemma~2.5]{LuYang2026}. Hence \(\varphi^*\cP^\circ\) has class \(D\alpha=0\) and is topologically trivial.
\end{proof}

The degree calculation differentiates the exact family, so a continuous trivialisation is insufficient. We use the following Lipschitz form of the regularisation argument of Lu and Yang \cite[Lemmas~3.7--3.8 and Theorem~3.9]{LuYang2026}.

\begin{lemma}\label{lem:lipschitz-trivialisation}
    Let \(M\) be a compact smooth manifold, let \(L\subseteq O(q)\) be compact, and let \(P\to M\) be a topologically trivial principal \(L\)-bundle with a finite Lipschitz atlas. Then \(P\) has a global Lipschitz section.
\end{lemma}

\begin{proof}
    Form the associated real vector bundle
    \[
         \cE=P\times_L\End(\R^q),
         \qquad (p\ell,T)\sim(p,\ell T),
    \]
    where \(L\) acts by left multiplication. The transition operators are Lipschitz in the base and orthogonal on the fibres. The map \(\iota\colon P\to\cE\), \(\iota(p)=[p, \Id]\), identifies \(P\) with the orbit subbundle \(P\times_L L\). Topological triviality gives a continuous section \(s_0\) of \(P\), and hence a continuous section \(\iota\circ s_0\) of \(\cE\).
    
    We record why this section admits a uniform Lipschitz approximation. Let \((U_i)_{i=1}^r\) be a finite trivialising cover and choose a smooth partition of unity \((\psi_i)_{i=1}^r\) with \(\operatorname{supp}\psi_i\Subset U_i\). Write \(\iota\circ s_0\) as a continuous map \(f_i\colon U_i\to\End(\R^q)\) in the \(i\)-th chart. By smooth approximation \cite[Theorem~6.21]{Lee2013}, choose a smooth \(g_i\) uniformly close to \(f_i\) on \(\operatorname{supp}\psi_i\). The local section represented by \(\psi_i g_i\) extends by zero to a global Lipschitz section: in another chart it is the product of a Lipschitz orthogonal transition operator and a compactly supported smooth map. Summing these sections gives a Lipschitz section \(\sigma\) of \(\cE\). Since the transition operators preserve the fibre norm,
    \[
         \norm{\sigma(x)-\iota(s_0(x))}
         \leqslant\sum_{i=1}^r\psi_i(x)\abs{g_i(x)-f_i(x)},
    \]
    so the approximation may be made uniformly as close as required.
    
    The compact subgroup \(L\) is an embedded submanifold of \(\End(\R^q)\). Left multiplication is orthogonal and carries normal spaces to normal spaces, so the tubular-neighbourhood construction gives an \(L\)-invariant neighbourhood \(\mathcal U\) of \(L\) and a smooth retraction \(r\colon\mathcal U\to L\) satisfying \(r(\ell T)=\ell r(T)\); see \cite[Theorem~6.24]{Lee2013}. Choose a smaller invariant tube \(\mathcal U'\) whose closure is compactly contained in \(\mathcal U\).
    
    We verify that the restriction of \(r\) is Lipschitz. Cover \(\overline{\mathcal U'}\) by finitely many convex open sets \(W_1,\ldots,W_s\) whose closures lie in \(\mathcal U\), and let
    \[
         M=\max_{1\leqslant j\leqslant s}\sup_{T\in W_j}\norm{Dr(T)}<\infty.
    \]
    Choose a Lebesgue number \(\eta>0\) for this cover of \(\overline{\mathcal U'}\). If \(T,T'\in\overline{\mathcal U'}\) and \(\abs{T-T'}<\eta\), then the two-point set \(\{T,T'\}\) lies in some \(W_j\); convexity places the segment between them in \(W_j\), and the mean-value estimate gives \(\abs{r(T)-r(T')}\leqslant M\abs{T-T'}\). If \(\abs{T-T'}\geqslant\eta\), then
    \[
         \abs{r(T)-r(T')}
         \leqslant\operatorname{diam}(L)
         \leqslant\frac{\operatorname{diam}(L)}{\eta}\abs{T-T'}.
    \]
    Thus \(r\) is Lipschitz on \(\overline{\mathcal U'}\), with constant at most \(\max\{M,\operatorname{diam}(L)/\eta\}\).
    
    Choose \(\sigma\) so close to \(\iota\circ s_0\) that all its fibre coordinates belong to \(\mathcal U'\). Equivariance makes \([p,T]\mapsto[p,r(T)]\) a well-defined fibrewise Lipschitz retraction onto \(\iota(P)\). Composing \(\sigma\) with this retraction and with \(\iota^{-1}\) gives a global Lipschitz section of \(P\).
\end{proof}

\begin{proposition}\label{prop:global-maps}
    There are a smooth map \(\varphi\colon S^{N-1}\to S^{N-1}\) of positive degree and a Lipschitz map \(z\mapsto A_z\) from \(S^{N-1}\) to \(\Hom_{\R}(E,V)\) such that
    \begin{equation}\label{eq:global-exact}
     A_z(E)=H_{\varphi(z)}
     \quad\hbox{and}\quad
     A_z(S)=K\cap H_{\varphi(z)}
    \end{equation}
    for every \(z\in S^{N-1}\).
\end{proposition}

\begin{proof}
    Pull \(\cP^\circ\) back by the map supplied by \Cref{lem:annihilation}. The pulled-back transition functions are Lipschitz because \(\varphi\) is smooth on a compact manifold, and the bundle is topologically trivial. By \Cref{lem:lipschitz-trivialisation}, it has a global Lipschitz section. In a pulled-back chart, this section is represented by \(A_z=A_i^\circ(\varphi(z))h_i(z)\), where \(A_i^\circ\) is a local exact section and \(h_i\) is a Lipschitz \(G^\circ\)-valued map. Hence \(z\mapsto A_z\) is locally Lipschitz. A finite-cover and Lebesgue-number argument makes it globally Lipschitz. Its values lie in the pullback of the exact-section bundle, and so \eqref{eq:global-exact} holds.
\end{proof}

\begin{proof}[Proof of \Cref{thm:spherical-intro}]
    For the family in \Cref{prop:global-maps}, exactness gives \(p_K(A_zy)=p_S(y)\), while \(A_zy\in H_{\varphi(z)}\subseteq\varphi(z)^\perp\) gives \((A_zy,\varphi(z))=0\). The real ambient dimension \(N\) is even, and \Cref{thm:exact-family-intro} therefore shows that \(p_S^2\) is a positive-definite quadratic form. A self-map of a sphere of nonzero degree is surjective. Thus every \(u\in S(V)\) is \(\varphi(z)\) for some \(z\), and \eqref{eq:global-exact} shows that \(K\cap H_u\) is linearly equivalent to the ellipsoid \(S\).
    
    Assume now the two-vector coverage property. Given \(x,y\in V\), choose \(u\) with \(x,y\in H_u\). The restriction of \(p_K\) to \(H_u\) is Minkowski's functional of the ellipsoid \(K\cap H_u\), and therefore
    \[
         p_K(x+y)^2+p_K(x-y)^2
         =2p_K(x)^2+2p_K(y)^2.
    \]
    In particular, the identity holds for every pair \(x,y\). Each vector also belongs to one of the sections, so \(p_K(-x)=p_K(x)\). The Jordan--von Neumann polarisation calculation, applied to \(q=p_K^2\), now shows that \(B(x,y)=\tfrac14(q(x+y)-q(x-y))\) is a symmetric real bilinear form and that \(B(x,x)=q(x)\) \cite{JordanVonNeumann1935}. It is positive definite because \(p_K\) is definite. Hence \(K=\{x:B(x,x)\leqslant1\}\) is an ellipsoid.
\end{proof}

The parity assumption on \(N\) is intrinsic to this argument. If \(N\) were odd and a nonzero-degree exact family existed for Minkowski's functional \(p\), the case \(k=0\) of \eqref{eq:signed-degree} would make \(p^N\) an odd-degree homogeneous polynomial which is positive away from zero. An odd-degree homogeneous polynomial changes sign under \(y\mapsto-y\), a contradiction.

\section{Complex and quaternionic applications}\label{sec:scalar-rigidity}

Before specialising the real theorem, let us record the relevant dimension count. Put \(\delta=\dim_{\R}\F\), temporarily allowing \(\F=\R\), and set \(m=\delta(n+1)\). For hyperplanes in an \((n+1)\)-dimensional \(\F\)-space, the unit-normal sphere is \(S^{m-1}\), the obstruction group in \Cref{lem:annihilation} is \(\pi_{m-2}(G^\circ)\), and \Cref{thm:exact-family-intro} produces the powers \(p_S^{m+2k}\). It is the homotopy-group index \(m-2\), not the degree of the self-map \(\varphi\), which must be positive and even; the degree need only be nonzero and is chosen positive in \Cref{lem:annihilation}. The real ambient dimension \(m\) must also be even. Over \(\R\), these requirements say that \(n\) is odd. Over \(\C\) and \(\Hh\), they hold for every \(n\geqslant2\): here \(m-2\) equals \(2n\) and \(4n+2\), respectively.

The quadratic form obtained from the real theorem must still be shown to respect the scalar structure. The following polarisation formula uses our right-module convention.

\begin{lemma}\label{lem:scalar-polarisation}
Let \(E\) be a right \(\F\)-vector space, and let \(q\) be a positive-definite real quadratic form on \(E_{\R}\) such that \(q(xa)=q(x)\abs a^2\) for \(x\in E\) and \(a\in\F\). Then there is a unique positive-definite \(\F\)-Hermitian form \(h\) with \(h(x,x)=q(x)\).
\end{lemma}

\begin{proof}
Let \(B\) be the symmetric real bilinear form which polarises \(q\). If \(\abs a=1\), then \(q(xa)=q(x)\), and polarisation gives \(B(xa,ya)=B(x,y)\). Thus right multiplication by \(a\) is orthogonal and has adjoint right multiplication by \(\bar a\). In particular, for each standard imaginary unit \(e\),
\(B(xe,y)=-B(x,ye)\).

Define
\[
 h(x,y)=
 \begin{cases}
 B(x,y)+B(xi,y)i, & \F=\C,\\
 B(x,y)+B(xi,y)i+B(xj,y)j+B(xk,y)k, & \F=\Hh.
 \end{cases}
\]
We verify the quaternionic case, the complex one being the same calculation with only \(i\). Write \(b_0=B(x,y)\), \(b_1=B(xi,y)\), \(b_2=B(xj,y)\), and \(b_3=B(xk,y)\). Since \(B(u,va)=B(u\bar a,v)\) for \(\abs a=1\), the multiplication rules give
\[
 h(x,yi)=-b_1+b_0i+b_3j-b_2k=h(x,y)i.
\]
The calculations for \(j\) and \(k\) are obtained by cyclic permutation. Real linearity then gives \(h(x,ya)=h(x,y)a\) for every quaternion \(a\). Symmetry of \(B\) and the skew-adjoint identity give \(B(ye,x)=-B(xe,y)\), so \(h(y,x)=\overline{h(x,y)}\). Right linearity in the second variable and Hermitian symmetry imply conjugate linearity in the first. Also \(B(xe,x)=0\), and hence \(h(x,x)=B(x,x)=q(x)>0\) for \(x\neq0\).

Finally, any Hermitian form with diagonal \(q\) has real part \(B\). Its coefficients along the standard imaginary units are recovered from the numbers \(B(xe,y)\), exactly as in the displayed definition of \(h\). This proves uniqueness. The formula is also a special case of \cite[Theorems~1.1 and~1.3]{BenderChakrabarti2023}.
\end{proof}

\begin{proof}[Proof of \Cref{thm:hyperplane-intro}]
Orient \(V_{\R}\) and choose an auxiliary \(\F\)-Hermitian inner product. For a unit vector \(u\in V_{\R}\), put
\[
 H_u=\{v\in V:\inner{u}{v}_{\F}=0\},
 \qquad
 \Pi_uv=v-u\inner{u}{v}_{\F}.
\]
The formula for \(\Pi_u\) shows that \(u\mapsto H_u\) is smooth. Each \(H_u\) has real dimension \(\delta n\), where \(\delta=\dim_{\R}\F\), and is contained in the real orthogonal complement of \(u\). The ambient real dimension \(N=\delta(n+1)\) is even. If \(x,y\in V\), their \(\F\)-linear span has dimension at most two. Its \(\F\)-orthogonal complement is nonzero because \(n\geqslant2\); choosing a unit vector \(u\) in this complement gives \(x,y\in H_u\). Thus \Cref{thm:spherical-intro} applies and shows that \(K\) is a real ellipsoid.

Let \(q=p_K^2\) be the resulting positive-definite real quadratic form. Since \(K\) is \(\F\)-circled, \(p_K(xa)=p_K(x)\) for \(\abs a=1\), and positive homogeneity gives \(q(xa)=q(x)\abs a^2\) for arbitrary \(a\in\F\). By \Cref{lem:scalar-polarisation}, \(q\) is the diagonal of a positive-definite \(\F\)-Hermitian form.
\end{proof}

\begin{proof}[Proof of \Cref{thm:functional-intro}]
Let \(x,y\in X\), and choose an \((n+1)\)-dimensional \(\F\)-subspace \(Y\) containing their \(\F\)-linear span. On an auxiliary Euclidean unit sphere in \(Y\), the definite function \(p\) is positive, continuous, and hence bounded away from zero. Thus \(K_Y=\{v\in Y:p(v)\leqslant1\}\) is an \(\F\)-circled star body, with positive Lipschitz radial function \(1/p\). Every \(\F\)-hyperplane of \(Y\) is an \(n\)-dimensional subspace of \(X\), and the maps in hypothesis~\emph{(ii)} identify the corresponding sections of \(K_Y\). By \Cref{thm:hyperplane-intro}, \(p|_Y\) is induced by an \(\F\)-Hermitian form. In particular, the parallelogram identity holds for \(x\) and \(y\).

Since \(x,y\) were arbitrary, \(q=p^2\) satisfies the parallelogram identity on \(X\). Absolute homogeneity makes \(q\) even and two-homogeneous, and \(q\geqslant0\). No further appeal to hypothesis~\emph{(i)} is needed. Indeed, the algebraic Jordan--von Neumann calculation makes
\[
 B(x,y)=\tfrac14\bigl(q(x+y)-q(x-y)\bigr)
\]
symmetric and additive in each variable. Applying non-negativity to \(q(u+tv)\) first for rational \(t\) gives \(\abs{B(u,v)}^2\leqslant q(u)q(v)\). Consequently, for fixed \(x,y\), the additive function \(t\mapsto B(tx,y)\) is bounded in absolute value by \(\abs t\,q(x)^{1/2}q(y)^{1/2}\); it is continuous and therefore real-linear. Thus \(B\) is a positive-definite real bilinear form with diagonal \(q\). Finally, \(q(xa)=q(x)\abs a^2\), and \Cref{lem:scalar-polarisation} gives the required \(\F\)-Hermitian form.
\end{proof}

\begin{proof}[Proof of \Cref{cor:normed-intro}]
A norm is Lipschitz on each finite-dimensional subspace with respect to any auxiliary Euclidean norm. If \(f\colon E\to F\) is one of the bijective metric isometries, the Mazur--Ulam theorem \cite{MazurUlam1932} shows that \(x\mapsto f(x)-f(0)\) is a real-linear isometry from \(E\) onto \(F\). Thus \Cref{thm:functional-intro} applies. Completion of the resulting Hermitian norm gives an \(\F\)-Hilbert space and changes nothing when \(X\) was complete.
\end{proof}

\section{Extensions to Fr\'echet spaces and beyond}\label{sec:frechet}

An isometric homogeneity condition on a Fr\'echet space requires more than its locally convex topology: one must fix either a defining sequence of seminorms or a compatible translation-invariant metric. Linear homeomorphism of finite-dimensional subspaces alone carries no information, since every finite-dimensional Hausdorff topological vector space over \(\F\) has its usual Euclidean topology. Completeness enters only after the inner-product structures have been found.

For an \(\F\)-seminorm \(p\), the associated normed quotient is \(X/\ker p\), with \(\norm{x+\ker p}_p=p(x)\). It need not be complete.

\begin{proof}[Proof of \Cref{thm:reduced-frechet}]
Assume (2), and let \((p_k)_{k\in\N}\) be a reduced \(n\)-isometric defining sequence. Put \(N_k=\ker p_k\) and \(d_k=\dim_{\F}(X/N_k)\). Since the sequence is increasing, \(N_{k+1}\subseteq N_k\), and the dimensions \(d_k\) are non-decreasing. We claim that \(d_k>n\) eventually. If not, every \(d_k\) is a finite integer at most \(n\), and the sequence eventually has a constant value. Nested subspaces of the same finite codimension are equal, so the kernels would then stabilise. Since the grading is Hausdorff, \(\bigcap_kN_k=\{0\}\); the stable kernel would be zero, giving \(\dim_{\F}X\leqslant n\), contrary to the hypothesis.

For every \(k\) in this cofinal tail, all \(n\)-dimensional subspaces of \(X/N_k\) are isometric. By \Cref{cor:normed-intro}, the quotient norm is induced by an \(\F\)-Hermitian form, and hence \(p_k\) is Hilbertian. Removing finitely many terms does not alter the topology. Let \(H_k=\widehat{X/N_k}\). Since \(p_k\leqslant p_{k+1}\), the canonical quotient map extends to a contraction \(\pi_{k+1,k}\colon H_{k+1}\to H_k\) with dense range. There is a canonical linear map \(J\colon X\to\varprojlim_{k\geqslant k_0}H_k\). It is injective because \(\bigcap_kN_k=\{0\}\), and the projective-limit seminorms pulled back by \(J\) are precisely the \(p_k\). Thus \(J\) is a topological embedding.

Let \(\xi=(\xi_k)_{k\geqslant k_0}\) belong to the inverse limit. For each \(k\geqslant k_0\), density of \(X/N_k\) in \(H_k\) gives \(x_k\in X\) such that \(\norm{x_k+N_k-\xi_k}_{H_k}<2^{-k}\). If \(j\leqslant k\), compatibility of \(\xi\) and contractivity of the bonding maps give \(\norm{x_k+N_j-\xi_j}_{H_j}<2^{-k}\). Hence, for fixed \(j\) and \(k,\ell\geqslant j\),
\[
 p_j(x_k-x_\ell)<2^{-k}+2^{-\ell}.
\]
The sequence \((x_k)\) is therefore Cauchy for every defining seminorm. Fr\'echet completeness gives an \(x\in X\) with \(x_k\to x\), and passage to the \(j\)-th coordinate yields \(x+N_j=\xi_j\). Thus \(Jx=\xi\), and \(J\) is onto.

Conversely, let \((p_k)\) be an increasing Hilbertian defining sequence. Each quotient \(X/\ker p_k\) is an inner-product space, although it need not be complete. Any two of its finite-dimensional subspaces of the same dimension are isometric. Hence the grading is reduced \(n\)-isometric whenever the quotient dimension exceeds \(n\).
\end{proof}

The conclusion need not be normable. Let \(s_{\F}\) be the space of rapidly decreasing \(\F\)-valued sequences, graded by the Hilbertian seminorms
\[
 p_m(x)^2=\sum_{j=1}^{\infty}j^{2m}\abs{x_j}^2.
\]
Its completion is the corresponding weighted \(\ell_2\)-space, so the grading is reduced \(n\)-isometric for every fixed \(n\). However, \(p_{m+1}(e_j)=j p_m(e_j)\). No single \(p_m\) induces the Fr\'echet topology. Thus the conclusion cannot in general be replaced by normability; once eventual Hilbertianity is known, the displayed inverse limit is the standard projective-limit representation.

\begin{proof}[Proof of \Cref{thm:simultaneous-collapse}]
Let us fix \(p\in\mathcal Q\). We first show that \(p\) is either zero or a norm. Suppose, to the contrary, that \(K=\ker p\) is nonzero and proper. Choose an \((n+1)\)-dimensional subspace \(Y\) which meets \(K\) nontrivially and is not contained in \(K\), and put \(L=Y\cap K\). Then \(1\leqslant r=\dim_{\F}L\leqslant n\). There is a hyperplane \(E\) of \(Y\) containing \(L\), and there is a hyperplane \(F\) of \(Y\) for which \(F\cap L\) is a hyperplane of \(L\). Thus \(\dim(E\cap K)=r\) and \(\dim(F\cap K)=r-1\). On the other hand, a \(p\)-similarity \(T\colon E\to F\) maps \(E\cap K\) bijectively onto \(F\cap K\), since \(p(Tx)=c_p(T)p(x)\) with \(c_p(T)>0\). This contradiction proves the claim.

For a nonzero \(p\), normalise each simultaneous similarity by replacing \(T\) with \(c_p(T)^{-1}T\). The resulting map is a \(p\)-isometry and remains a similarity for every member of \(\mathcal Q\). Hence all \(n\)-dimensional subspaces of \((X,p)\) are isometric. If \(\F\in\{\C,\Hh\}\), \Cref{cor:normed-intro} shows that \(p\) is induced by a positive-definite \(\F\)-Hermitian form. If \(\F=\R\), apply Gromov's theorem when \(n\) is even and the Lu--Yang theorem when \(n\) is odd to every \((n+1)\)-dimensional subspace \cite{Gromov1967,LuYang2026}; the resulting parallelogram identities again show that \(p\) is Hilbertian.

Let \(p,q\in\mathcal Q\) be nonzero, and restrict them to an arbitrary \((n+1)\)-dimensional subspace \(Y\). Relative to the Hermitian inner product inducing \(p\), there is a positive self-adjoint \(\F\)-linear operator \(R\) such that \(q(x)^2=\inner{x}{Rx}_p\). The spectral theorem over \(\R\), \(\C\), and \(\Hh\) gives a \(p\)-orthonormal eigenbasis \(e_1,\ldots,e_{n+1}\), with eigenvalues \(0<\lambda_1\leqslant\cdots\leqslant\lambda_{n+1}\). Let \(E_i\) be the hyperplane spanned by all basis vectors except \(e_i\).

Let \(U\colon E_i\to E_j\) be a simultaneous similarity, normalised to be a \(p\)-isometry, and write \(q(Ux)=c q(x)\). Then \(U^*(R|_{E_j})U=c^2R|_{E_i}\). Thus the eigenvalue multiset of \(R|_{E_j}\) is proportional to that of \(R|_{E_i}\). It follows that all one-point deletion multisets of \(\{\lambda_1,\ldots,\lambda_{n+1}\}\) are proportional.

Compare first the deletion of \(\lambda_1\) with the deletion of \(\lambda_{n+1}\). Since both remaining lists are ordered, there is \(c\geqslant1\) such that \(\lambda_{r+1}=c\lambda_r\) for \(1\leqslant r\leqslant n\). If \(c>1\), the deletion lists for \(\lambda_1\) and \(\lambda_2\) have condition numbers \(c^{n-1}\) and \(c^n\), respectively. Proportional positive multisets have the same condition number, so this is impossible. Hence \(c=1\), all the eigenvalues are equal, and \(q|_Y=a_Yp|_Y\) for some \(a_Y>0\).

Let us fix \(x_0\neq0\). For any \(y\in X\), choose an \((n+1)\)-dimensional subspace \(Y\) containing \(x_0\) and \(y\). The preceding paragraph gives \(a_Y=q(x_0)/p(x_0)\), independently of \(Y\), and consequently \(q=ap\) on \(X\). Thus all nonzero members of \(\mathcal Q\) are proportional. Since \(\mathcal Q\) is separating, it contains a nonzero member, and the topology generated by \(\mathcal Q\) is precisely its norm topology. If this topology is complete, that Hermitian norm is complete.
\end{proof}

A compatible translation-invariant metric determines the family of its centred balls, and every linear \(d\)-isometry preserves each ball together with Minkowski's functional associated with it.

\begin{proof}[Proof of \Cref{thm:metric-frechet}]
Let us fix \(x_0\neq0\) and replace \(\mathcal R\) by \(\{r\in\mathcal R:r<d(x_0,0)\}\). It still has zero in its closure, and every corresponding ball is proper. Each \(B_r^d\) is closed and contains the open ball of radius \(r\), so it is a neighbourhood of zero and hence is absorbing. If \(T\colon E\to F\) is one of the linear \(d\)-isometries, then \(T0=0\) and
\[
 T(B_r^d\cap E)=B_r^d\cap F
 \qquad(r\in\mathcal R).
\]
Thus the same map preserves Minkowski's functional associated with every ball considered below.

Suppose first that the balls are \(\F\)-absolutely convex. For \(r\in\mathcal R\), put \(p_r(x)=\inf\{t>0:x\in tB_r^d\}\). Since \(B_r^d\) is closed, absolutely convex, and absorbing, \(p_r\) is a continuous \(\F\)-seminorm and \(B_r^d=\{x:p_r(x)\leqslant1\}\). The family \(\{p_r:r\in\mathcal R\}\) is separating: if \(x\neq0\), choose \(r<d(x,0)\), so that \(x\notin B_r^d\). It generates the original topology because the balls with \(r\in\mathcal R\) form a neighbourhood basis at zero. Every linear \(d\)-isometry preserves every \(p_r\). The case \(c_p(T)=1\) of \Cref{thm:simultaneous-collapse} shows that the nonzero \(p_r\)'s are positive multiples of a single Hermitian norm \(q\). Consequently, \(q\) induces the original Fr\'echet topology and is complete.

Suppose next that the finite-dimensional sections are Lipschitz star bodies, and define \(p_r\) by the same infimum. If \(W\) is finite-dimensional, then \(p_r|_W\) is Minkowski's functional of \(B_r^d\cap W\). Taking \(W=\operatorname{span}_{\F}\{x\}\) shows that \(p_r(x)\) is finite and positive for \(x\neq0\), and the circled radial structure gives \(p_r(xa)=p_r(x)\abs a\). Closedness gives \(B_r^d=\{x:p_r(x)\leqslant1\}\). The radial regularity says precisely that \(p_r\) is Lipschitz on an auxiliary Euclidean sphere in every finite-dimensional subspace. Since every linear \(d\)-isometry preserves \(p_r\), \Cref{thm:functional-intro} makes each \(p_r\) a Hermitian norm when \(\F\in\{\C,\Hh\}\). When \(\F=\R\), the additional assumption that \(n\) is odd allows the same conclusion by applying \Cref{thm:spherical-intro} on every \((n+1)\)-dimensional subspace and then using the parallelogram identity. The family is separating as above, and the field-independent \Cref{thm:simultaneous-collapse} shows that its members are proportional. Their unit balls form a neighbourhood basis, so the resulting Hermitian norm induces the original topology and is complete.

Finally, suppose that every closed centred \(d\)-ball is \(\F\)-absolutely convex. Apply the first part to Minkowski's functional of each proper ball. There is a fixed Hermitian norm \(q\) for which every proper \(B_r^d\) is a scalar multiple of the \(q\)-unit ball; nonproper balls are all of \(X\). Vectors with the same \(q\)-norm therefore belong to exactly the same closed \(d\)-balls. Since \(d(x,0)=\inf\{r\geqslant0:x\in B_r^d\}\), they have the same distance from zero. Hence \(d(x,y)=\phi(q(x-y))\) for a non-decreasing function \(\phi\) which vanishes only at zero. If \(e\) is a \(q\)-unit vector, translation invariance and the triangle inequality give
\[
 \phi(s+t)=d((s+t)e,0)
 \leqslant d(se,0)+d(te,0)
 =\phi(s)+\phi(t),
\]
so \(\phi\) is subadditive.
\end{proof}

The Fr\'echet hypothesis is used only for completeness. For a Hausdorff metrisable locally convex space satisfying the other assumptions, the same argument produces a Hermitian norm which induces the topology, and the completion is Hilbert. The real version is included in \Cref{thm:metric-frechet}: absolutely convex balls give the conclusion for every \(n\geqslant2\), using the real Banach theorem at each level, whereas the star-body alternative is asserted for odd \(n\), the parity range of \Cref{thm:spherical-intro}.

The convex-ball hypothesis occurs for a standard class of compatible metrics. Let \((q_j)_{j\in\N}\) be a separating increasing sequence of \(\F\)-seminorms defining a Fr\'echet topology, and put
\[
 d_0(x,y)=\sup_{j\geqslant1}2^{-j}
 \frac{q_j(x-y)}{1+q_j(x-y)}.
\]
The function \(t\mapsto t/(1+t)\) is increasing and subadditive, so \(d_0\) is a translation-invariant metric. It induces the original topology: the finitely many terms with small index control any prescribed finite collection of seminorms, while the remaining terms are uniformly bounded by \(2^{-j}\). The same finite-head and uniform-tail argument shows completeness: a \(d_0\)-Cauchy sequence is Cauchy for each \(q_j\), hence converges in the Fr\'echet space, and then converges for \(d_0\). Moreover,
\[
 B_r^{d_0}
 =\bigcap_{2^{-j}>r}
 \left\{x:q_j(x)\leqslant\frac{r}{2^{-j}-r}\right\},
\]
so every closed ball is \(\F\)-absolutely convex, and any linear map preserving every \(q_j\) preserves \(d_0\). The theorem nevertheless concerns the chosen metric: invariance of the Fr\'echet topology alone does not imply invariance of its centred balls. \\

\noindent\textbf{Acknowledgements.} The first-named author acknowledges funding from the EPSRC (grant number EP/W524438/1) that has supported his studies. The second-named author acknowledges with thanks support of the Czech Academy of Sciences in Prague (RVO: 67985840). \\

\noindent\textbf{AI usage statement.} OpenAI's ChatGPT 5.6 Sol was used during the development of this work. After reading the recent article of Lu and Yang \cite{LuYang2026}, the authors observed that the techniques developed there could be extended to substantially more general settings. ChatGPT 5.6 Sol assisted with certain technical details needed to carry out these extensions.

\end{document}